\documentclass[11pt]{article}
\usepackage[letterpaper]{geometry}
\usepackage{amsfonts}
\usepackage{bbm}
\usepackage{amsmath,booktabs,ctable,threeparttable}
\usepackage{amssymb,amsfonts,boxedminipage}
\usepackage{amsthm}
\usepackage{algorithm}
\usepackage{algorithmic}
\usepackage{epsfig,graphicx,picins,picinpar,subfigure}
\usepackage{multirow}
\usepackage{verbatim}
\usepackage{bm}
\usepackage[round,numbers,authoryear]{natbib}
\usepackage[colorlinks,
linkcolor=red,
anchorcolor=green,
citecolor=blue
]{hyperref}

\numberwithin{equation}{section}

\newtheorem{theorem}{Theorem}

\newtheorem{definition}[theorem]{Definition}

\newtheorem{lemma}[theorem]{Lemma}

\newtheorem{assumption}[theorem]{Assumption}

\newcommand{\mrd}{\mathrm{d}}
\newcommand{\var}{\mathrm{Var}}
\newcommand{\unif}{\mathrm{Unif}}
\newcommand{\cov}{\mathrm{Cov}}

\begin{document}
\title{Asymptotic Normality of Extensible Grid Sampling}
    \author{Zhijian He\footnote{Corresponding Author. Email: hezhijian87@gmail.com}\\Lingnan (University) College, Sun Yat-sen University
    	\and
    	Lingjiong Zhu\\Department of Mathematics, Florida State University}
\maketitle

\begin{abstract}
Recently, \cite{he:owen:2016} proposed the use of Hilbert's space filling curve (HSFC) in numerical integration as a way of reducing the dimension from $d>1$ to $d=1$. This paper studies the asymptotic normality of the HSFC-based estimate when using scrambled van der Corput sequence as input. We show that the estimate has an asymptotic normal distribution for functions in $C^1([0,1]^d)$, excluding the trivial case of constant functions. The asymptotic normality also holds for  discontinuous functions under mild conditions. It was previously known only that scrambled $(0,m,d)$-net  quadratures  enjoy the asymptotic normality for smooth enough functions, whose mixed partial gradients satisfy a H\"older condition. As a by-product, we find lower bounds for the variance of the HSFC-based estimate. Particularly, for nontrivial functions in $C^1([0,1]^d)$, the low bound is of order $n^{-1-2/d}$, which matches the rate of the upper bound  established in \cite{he:owen:2016}.

\smallskip
\noindent \textbf{Keywords:} asymptotic normality; Hilbert's space filling curve; van der Corput sequence; randomized quasi-Monte Carlo; extensible grid sampling
\end{abstract}

\section{Introduction}
Quasi-Monte Carlo (QMC) sampling has gained increasing popularity in numerical integration over the unit cube $[0,1]^d$ (see, e.g., \cite{lecu:2009, dick:2010, dick:2013}). It is known that functions with finite variation in the sense of Hardy and Krause can be integrated with an error of $O(n^{-1}(\log n)^d)$, compared to $O(n^{-1/2})$ for ordinary Monte Carlo sampling; see \cite{nied:1992} for details.

In this paper, we consider an alternative numerical integration based on Hilbert's space filling curve (HSFC) as introduced in \cite{he:owen:2016}.
An HSFC is a continuous mapping $H(x)$ from $[0,1]$ to $[0,1]^d$ for $d\geq1$. We take the convention that $H(x)=x$ for $d=1$. Formally, the Hilbert curve is defined through the limit of  a series of recursive curves.
An illustration of the generative process of the HSFC
with increasing recursion order for $d=2$ is presented in Figure~\ref{fig:hsfc}. For detailed definitions and the properties of the HSFC, we refer to \cite{he:owen:2016}. Let us consider the problem of estimating an integral over the $d$-dimensional unit cube $[0,1]^d$:
\begin{equation}\label{eq:problem}
\mu=\int_{[0,1]^d}f(X) \mrd X.
\end{equation}
The HSFC-based estimate takes the form
\begin{equation}\label{eq:hilbertest}
\hat{\mu}_{n}=\frac 1 n \sum_{i=1}^n f(H(x_i)),
\end{equation}
where $x_i$ are some well-chosen points in $[0,1]$. In this paper, we focus on the case of using van der Corput sequence (in base $b\ge 2$) with the nested uniform scrambling of \cite{owen:1995} as the inputs $x_i$, for which the estimate is extensible and unbiased. The nested uniform scrambling method is a kind of randomization techniques used commonly in randomized QMC; see \cite{owen:1995} for details and \cite{lecu:2002} for a survey of various randomized QMC methods. \cite{mato:1998} proposed a random linear scrambling method that does
not require as much randomness and storage.
\cite{he:owen:2016} found convergence rates of the extensible estimate for  functions that are  Lipschitz continuous or piecewise Lipschitz continuous. More precisely, for Lipschitz continuous functions, they derived a root mean-squared error (RMSE) of $O(n^{-1/2-1/d})$. For the piecewise Lipschitz continuous functions, an RMSE of $O(n^{-1/2-1/(2d)})$ is obtained. \cite{schr:2016} compared  the star discrepancies and RMSEs of using the van der Corput and the golden ratio generator sequences. 

\begin{figure}[ht]
	\includegraphics[width=\hsize]{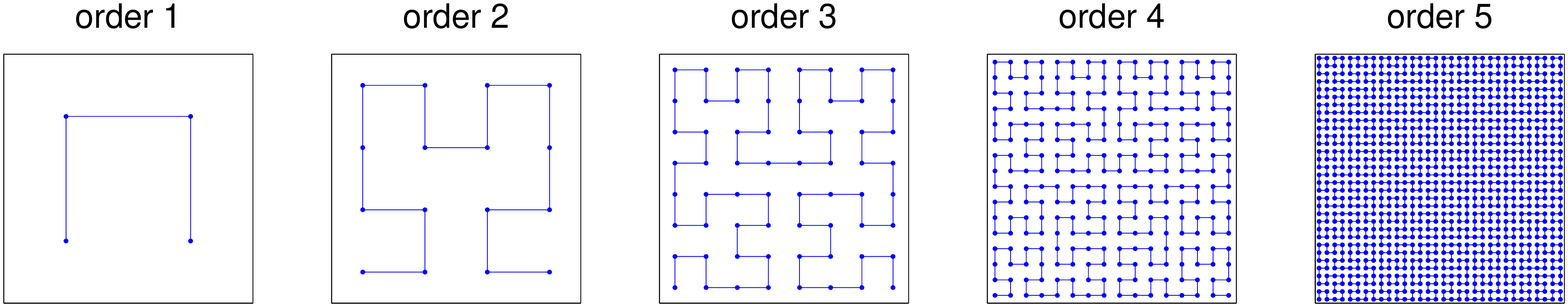}
	\caption{First five steps of the recursive construction of the HSFC for $d=2$.}\label{fig:hsfc}
\end{figure}

Actually, the upper bounds of the RMSE do not tell much about the error distribution. It is often of interest to obtain asymptotically valid confidence interval type guarantees as the usual Monte Carlo sampling. The central limit theorem (CLT) is invoked routinely to compute a confidence
interval on the estimate based on a normal approximation. Indeed, \cite{loh:2003} showed that the  nested uniform scrambled $(0,m,d)$-net (in base $b$) estimate has an
asymptotic normal distribution for smooth enough functions. Recently, by developing on the work by \cite{loh:2003}, \cite{basu:2016} showed that the scrambled geometric net estimate has an
asymptotic normal distribution for certain smooth functions defined
on products of suitable subsets of $\Re^d$. However, in most cases, 
the randomly-shifted lattice rule (another branch of randomized QMC techniques) estimates may be far
from the normally distributed ones; see \cite{lecu:2010} for discussions and examples.

In this paper we study the asymptotic normality of the HSFC-based estimate \eqref{eq:hilbertest} with sample sizes $n=b^m$, $m=0,1,2,\dots$. In practice, we often choose $b=2$ because the base $b=2$ is the same as used to approximate (and define) the Hilbert curve; see \cite{{butz:1971}} for the algorithm. The main contribution of our paper is two fold. First, for nontrivial functions in $C^1([0,1]^d)$, we establish a lower bound on $\var(\hat{\mu}_{n})$ which matches the upper bound $O(n^{-1-2/d})$ found in \cite{he:owen:2016}. A similar lower bound is established for piecewise smooth functions. Second, we prove that the asymptotic normality of the HSFC-based estimate $\hat{\mu}_{n}$ holds for  three classes of functions. In other words, we show that
$$\frac{\hat{\mu}_n-\mu}{\sqrt{\var(\hat{\mu}_{n})}}\to N(0,1)$$
in distribution as $n=b^m\to \infty$.
These results can be applied to stratified sampling on a regular grid with sample sizes $n=m^d$, but the  HSFC-based estimate we study, does not require the highly composite sample sizes that the grid sampling requires, particularly for large $d$. The main idea to prove the asymptotic normality is based on the Lyapunov CLT (see, e.g., \cite{chung:2001}), which is quite different from the techniques used in  \cite{loh:2003} and \cite{basu:2016}. Our proofs do not rely on the upper bounds established in \cite{he:owen:2016}.

The rest of the paper is organized as follows. In Section~\ref{sec:grid}, we study the asymptotic normality of stratified sampling on a regular grid with $n=m^d$, which can be viewed as a special case of the HSFC-based estimate \eqref{eq:hilbertest} using scrambled van der Corput sequence. In Section~\ref{sec:main}, we give lower bounds on $\var(\hat{\mu}_{n})$ and establish  asymptotic normality of the estimate $\hat{\mu}_{n}$ for three cases of integrands. In Section~\ref{sec:numer}, we give some empirical verifications on asymptotic normality for the HSFC sampling and other competitive QMC methods.  Section~\ref{eq:final} concludes this paper.

\section{Grid-based Stratified Sampling}\label{sec:grid}
In this section, we consider the regular grid sampling with sample sizes $n=m^d$.
The $d$-dimensional unit cube $[0,1]^d$ can be split into $m^d$
congruent subcubes with sides of length $1/m$, say, $E_i$, $i=1,\dots,n$. The grid-based stratified estimate of the integral \eqref{eq:problem}  is given by
\begin{equation}\label{eq:est}
\tilde{\mu}_n=\frac{1}{n}\sum_{i=1}^{n}f(U_{i}),
\end{equation}
where $U_i\sim \unif(E_i)$ independently.
Denote $\nabla f(X)=(\frac{\partial f(X)}{\partial X_1} ,\dots,\frac{\partial f(X)}{\partial X_d})^\top$ as the gradient vector of $f(X)$, and let $\Vert\cdot\Vert$ be the usual Euclidean norm. The next lemma discusses the variance of $\tilde{\mu}_n$, which was proved in \cite{owen:2013}. We prove it here also, because we make extensive use of that result. 
\begin{lemma}\label{lem:asymvar}
	Assume that $f(X)\in C^1([0,1]^d)$.  Then 
	\begin{equation}\label{eq:asyVar}
	\lim_{n\rightarrow\infty}n^{1+\frac{2}{d}}\var(\tilde{\mu}_n)
	=\frac{1}{12}\int_{[0,1]^{d}}\Vert\nabla f(X)\Vert^{2}\mrd X,
	\end{equation}
	where the limit is taken through values $n=m^d$ as $m\to \infty$.
\end{lemma}
\begin{proof}
	Note that $U_i$ is uniformly distributed within the cube $E_i$ with sides of length $1/m$. Let $c_i$ be the center of $E_i$. Since $f\in C^1([0,1]^d)$, the first-order Taylor approximation gives 
	\begin{equation}
	f(U_i)=L_i+R_i,
	\end{equation}
	where $L_i=f(c_i)+\nabla f(c_i)^\top(U_i-c_i)$ and $R_i=o(1/m)$.
	For the linear term $L_i$, we have 
	$$\var(L_i)=\frac 1{12m^2}\sum_{i=1}^d \left(\frac{\partial }{\partial X_i}f(c_i)\right)^2=\frac 1{12m^2}\Vert\nabla f(c_i)\Vert^2,$$
	since $U_i-c_i\sim \unif [-1/(2m),1/(2m)]^d$. For the error term $R_i$, we have $\var(R_i)=o(1/m^2)$. Also, $\cov(L_i,R_i)=o(1/m^2)$. As a result,
	\begin{align*}
	\var(\tilde{\mu}_n)&=\frac{1}{n^2}\sum_{i=1}^n\var(f(U_i))\\
	&=\frac{1}{n^2}\sum_{i=1}^n\var(L_i)+o\left(\frac{1}{nm^2}\right)\\
	&=\frac{1}{12n^2m^2}\sum_{i=1}^n\Vert\nabla f(c_i)\Vert^2 + o\left(\frac{1}{nm^2}\right).
	\end{align*}
	Since $m=n^{1/d}$ and
	$$\lim_{n\to \infty}\frac 1 n\sum_{i=1}^n\Vert\nabla f(c_i)\Vert^2=\int_{[0,1]^{d}}\Vert\nabla f(X)\Vert^{2}\mrd X,$$
	we conclude that \eqref{eq:asyVar} holds.
\end{proof}

\begin{theorem}\label{thm:gridsCLT}
	If $f(X)\in C^1([0,1]^d)$ and $\sigma^2=(1/12)\int_{[0,1]^{d}}\Vert\nabla f(X)\Vert^{2}\mrd X>0$, then
	\begin{equation}\label{eq:gridsCLT}
	\frac{\tilde{\mu}_n-\mu}{\sigma n^{-1/2-1/d}}\to N(0,1),
	\end{equation}
	in distribution as $n=m^d\to\infty$.
\end{theorem}
\begin{proof}
Since $f\in C^1([0,1]^d)$, $f$ is Lipschitz continuous.  Then, for any $\delta>0$,
\begin{equation}\label{eq:bound2pdelta}
\mathbb{E}[\vert f(U_{i})-\mathbb{E}[f(U_{i})]\vert^{2+\delta}]
\leq C_{d,\delta}n^{-\frac{2+\delta}{d}},
\end{equation}
where $C_{d,\delta}>0$ is some constant that only depends on $d$ and $\delta$, and we used the fact that the diameter of $E_i$ is  $\sqrt{d}n^{-1/d}$.

Let $s_{n}^{2}=\sum_{i=1}^{n}\sigma_{i}^{2}$, where
$\sigma_{i}^{2}=\var(f(U_{i}))$. From Lemma~\ref{lem:asymvar}, we have
\begin{equation}
\lim_{n\rightarrow\infty}\frac{s_{n}^{2}}{n^{1-\frac{2}{d}}}=\frac{1}{12}\int_{[0,1]^{d}}\Vert\nabla f(X)\Vert^{2}\mrd X>0.
\end{equation}
Therefore, the Lyapunov condition
\begin{align*}
\lim_{n\rightarrow\infty}\frac{1}{s_{n}^{2+\delta}}\sum_{i=1}^{n}\mathbb{E}[\vert f(U_{i})-\mathbb{E}[f(U_{i})]\vert^{2+\delta}]
&\leq\limsup_{n\rightarrow\infty}C_{d,\delta}
\frac{n^{1-\frac{2+\delta}{d}}}{s_{n}^{2+\delta}}
\\&=\limsup_{n\rightarrow\infty}C_{d,\delta}n^{-\frac{\delta}{2}}\left(\frac{n^{1-\frac 2 d}}{s_n^2}\right)^{\frac{2+\delta}{2}}=0
\end{align*}
is satisfied. Using the Lyapunov CLT, we get \eqref{eq:gridsCLT}.
\end{proof}
To avoid the trivial case of constant functions (which yields an identically
zero variance), Theorem~\ref{thm:gridsCLT} assumes that $\int_{[0,1]^{d}}\Vert\nabla f(x)\Vert^{2}\mrd x>0$. Theorem~\ref{thm:gridsCLT} admits the asymptotic normality of $\tilde{\mu}_n$. The grid-based stratified estimate has variance $O(n^{-1-2/d})$, compared to Monte Carlo variance $O(n^{-1})$. This actually holds for Lipschitz continuous functions covering the class of functions $C^1([0,1]^d)$ in Theorem~\ref{thm:gridsCLT}.

\section{HSFC-based Sampling}\label{sec:main}
In this section, we study the HSFC-based estimate given by \eqref{eq:hilbertest}, where $x_i$ are the first $n=b^m$ points of the scrambled van der Corput sequence in base $b\ge 2$. 
Let $a_i$ be  the first $n$ points of the van der Corput sequence in base $b$ \cite{vdc:1935}. The integer $i-1\ge 0$ is written in  base $b$ as $i-1=\sum_{j=1}^\infty a_{ij} b^{j-1}$ for $a_{ij}\in\{0,\dots,b-1\}$. Then $a_i$ is then defined by
$$a_i=\sum_{j=1}^\infty a_{ij}b^{-j}.$$ The scrambled version of $a_1,\dots,a_n$ is $x_1,\dots,x_n$ written as $x_i = \sum_{j=1}^{\infty} x_{ij}b^{-j}$, where $x_{ij}$ are defined through random permutations of the $a_{ij}$. These permutations depend on $a_{ik}$, for $k<j$. More precisely, $x_{i1}=\pi(a_{i1})$, $x_{i2}=\pi_{a_{i1}}(a_{i2})$ and generally for $j\ge 2$
$$x_{ij}=\pi_{a_{i1}\dots a_{ij-1}}(a_{ij}).$$
Each random permutation is uniformly distributed over the $b!$ permutations of $\{0,\dots,b-1\}$, and the permutations are mutually independent.

In this setting,  thanks to the nice property of the nested uniform scrambling,  the data values in the scrambled sequence can be reordered such that $x_i\sim \unif(I_i)$ independently with $I_{i}$ for $i=1,\dots,b^m$. Let $E_{i}=H(I_{i})$. As used in \cite{he:owen:2016}, the estimate \eqref{eq:hilbertest} can be rewritten as
\begin{equation*}
\hat{\mu}_n = \frac 1n\sum_{i=1}^n f(X^{(i)}),
\end{equation*}
where $X^{(i)}=H(x_i)\sim\unif(E_i)$.  
This implies that the HSFC-based sampling is actually a stratified sampling because $\{E_i\}_{i=1}^n$ is a split of $[0,1]^d$. Figure~\ref{fig:grids} illustrates such splits of $[0,1]^2$ when $b=2$. 
\cite{he:owen:2016} proved the unbiasedness of $\hat{\mu}_{n}$ for any $f\in L^2([0,1]^d)$ and gave some upper bounds for  $\var(\hat{\mu}_{n})$ under certain assumptions on the class of
integrands $f$. Their proofs make use of the properties of the HSFC presented in the next lemma, which are also important in studying the asymptotic normality of the HSFC-based sampling. Denote $\lambda_d(\cdot)$ as the Lebesgue measure on $\Re^d$. 
\begin{lemma}\label{lem:ei}
	Let $A=H([p,q])$ for $0\leq p<q\le 1$. Then $\lambda_d(A)=\lambda_1([p,q])=q-p$. If $x\sim \unif([p,q])$, then $H(x)\sim \unif(A)$. Let $r$ be the diameter of $A$. Then $r\leq 2\sqrt{d+3}(q-p)^{1/d}$.
\end{lemma}
\subsection{Smooth Functions}\label{sec:smooth}
\cite{loh:2003} and \cite{basu:2016} focused on smooth functions whose mixed partial gradient satisfies a H\"older condition, which was first studied in \cite{owen:1997}. 
Here we work with a weaker smoothness condition,  in the sense that $f(X)\in C^1([0,1]^d)$ as required in Theorem~\ref{thm:gridsCLT}.

\begin{theorem}\label{thm:HilbertCLTsmooth}
	Assume that $f(X)\in C^1([0,1]^d)$ and $\sigma^2=\int_{[0,1]^{d}}\Vert\nabla f(X)\Vert^{2}\mrd X>0$. Then for all sufficiently large $n$, we have
	\begin{equation}\label{eq:lbsmooth}
\var(\hat{\mu}_{n})\geq \frac{\sigma^2}{96}2^{-2/d-d}n^{-1-2/d}.
	\end{equation}
	Also,
	\begin{equation}\label{eq:HilbertCLT}
	\frac{\hat{\mu}_{n}-\mu}{\sqrt{\var(\hat{\mu}_{n})}}\to N(0,1),
	\end{equation}
	in distribution as $n=b^m\to\infty$.
\end{theorem}
\begin{proof}
Let $m = \lceil\frac{\log_2n+1}{d}\rceil$. Then there exists an interval $J_i$ of the form $[(k-1)/2^{dm},k/2^{dm}]$ such that $J_i\subset I_i$. This is because $\lambda_1(I_i)\geq 2\lambda_1(J_i)$. Let $s_{n}^{2}=\sum_{i=1}^{n}\sigma_{i}^{2}$, where
$\sigma_{i}^{2}=\var(f(X^{(i)}))$. Let $\mu_i=\mathrm{E}(f(X^{(i)}))$, and let $\tilde{E}_i=H(J_i)$. Note that  $\tilde{E}_i\subset E_i$ due to  $J_i\subset I_i$. Let $\mu'_i=\tilde{\mathrm{E}}(f(X^{(i)}))$, where the expectation is taken from $X^{(i)}\sim \unif(\tilde{\mathrm{E}}_i)$. Based on some basic algebra, we find
\begin{align*}
\sigma_{i}^{2}&=\var(f(X^{(i)}))=\frac{1}{\lambda_d(E_i)}\int_{E_i}[f(X)-\mu_i]^2 \mrd X \\
&\ge \frac{1}{\lambda_d(E_i)}\int_{\tilde{E}_i}[f(X)-\mu_i]^2 \mrd X\\
&= \frac{1}{\lambda_d(E_i)}\int_{\tilde{E}_i}\left([f(X)-\mu'_i]^2+2(f(X)-\mu'_i)(\mu'_i-\mu_i)+ (\mu'_i-\mu_i)^2\right)\mrd X\\
&=\frac{\lambda_d(\tilde{E}_i)}{\lambda_d(E_i)}\left(\widetilde{\mathrm{Var}}(f(X^{(i)}))+(\mu'_i-\mu_i)^2\right)\\
&\ge 2^{-(1+d)}\widetilde{\mathrm{Var}}(f(X^{(i)})),
\end{align*}
where $\widetilde{\mathrm{Var}}$ is taken  over $X^{(i)}\sim\unif(\tilde{E}_i)$, and we used two results by applying Lemma~\ref{lem:ei} that $\lambda_d(E_i)=1/n$ and
\begin{equation}\label{eq:1}
\lambda_d(\tilde{E}_i)=\lambda_1(J_i)=\frac{1}{2^{dm}}\ge\frac{1}{ 2^{1+d}n}=2^{-(1+d)}\lambda_d(E_i).
\end{equation} 
We thus have
\begin{equation}
\frac{s_{n}^{2}}{n^{1-\frac{2}{d}}}\geq \frac{1}{2^{1+d}n^{1-\frac{2}{d}}}\sum_{i=1}^n\widetilde{\mathrm{Var}}(f(X^{(i)}))=:K_1(n).
\end{equation}

Notice that $\tilde{E}_i$ is a cube with sides of length $2^{-m}$. Following the proof of Theorem~\ref{thm:gridsCLT}, we have
\begin{equation}\label{eq:tvar}
\widetilde{\mathrm{Var}}(f(X^{(i)})) = \frac{1}{12\cdot 2^{2m}}\Vert\nabla f(c_i)\Vert^{2}+o(2^{-2m}),
\end{equation}
where $c_i$ is the center of $\tilde{E}_i$.
Therefore, 
\begin{align}
\liminf_{n\to\infty} K_1(n)&=\liminf_{n\to\infty}\frac{1}{2^{1+d}n^{1-\frac{2}{d}}}\left(\frac{1}{12\cdot 2^{2m}}\sum_{i=1}^n \Vert\nabla f(c_i)\Vert^{2}+o(2^{-2m}n)\right)\notag\\
&\geq \liminf_{n\to \infty} \frac{1}{2^{1+d}n^{1-2/d}}\frac{1}{12\cdot 2^{2(1+d)/d}n^{2/d}}\sum_{i=1}^n \Vert\nabla f(c_i)\Vert^{2}\label{eq:lb1}\\
&=\frac{1}{96\cdot 2^{2/d+d}}\int_{[0,1]^{d}}\Vert\nabla f(X)\Vert^{2}\mrd X>0.\label{eq:lb2}
\end{align}
The inequality \eqref{eq:lb1} is due to $m\le(\log_2 n+1)/d+1$. The equality \eqref{eq:lb2} is due to $c_i\in E_i$ and $\{E_1,\dots,E_n\}$ is a split of $[0,1]^d$.
As a result,
\begin{equation}\label{eq:lowerbound}
\liminf_{n\to\infty} \frac{s_{n}^{2}}{n^{1-\frac{2}{d}}}\ge \liminf_{n\to\infty} K_1(n)\ge\frac{1}{96\cdot 2^{2/d+d}}\int_{[0,1]^{d}}\Vert\nabla f(X)\Vert^{2}\mrd X>0.
\end{equation}
Combing \eqref{eq:lowerbound} with $\var(\hat{\mu}_{n})=s_n^2/n^2$ establishes the inequality \eqref{eq:lbsmooth}.

Similar to \eqref{eq:bound2pdelta}, for any $\delta>0$, there exists a constant $C_{d,\delta}$ depending on $d$ and $\delta$ such that
\begin{equation}\label{eq:upperbound}
\mathbb{E}\left[\left|f(X^{(i)})-\mathbb{E}[f(X^{(i)})]\right|^{2+\delta}\right]
\leq C_{d,\delta}n^{-\frac{2+\delta}{d}},
\end{equation}
because the diameter of $E_i$ is not larger than $2\sqrt{d+3}n^{-1/d}$ by Lemma~\ref{lem:ei}. 
Using \eqref{eq:lowerbound} and \eqref{eq:upperbound}, the Lyapunov condition
\begin{align*}
\lim_{n\rightarrow\infty}\frac{1}{s_{n}^{2+\delta}}\sum_{i=1}^{n}\mathbb{E}\left[\left|f(X^{(i)})
-\mathbb{E}[f(X^{(i)})]\right|^{2+\delta}\right]
&\leq\limsup_{n\rightarrow\infty}\frac{C_{d,\delta}n^{1-\frac{2+\delta}{d}}}{s_{n}^{2+\delta}}\\
&=\limsup_{n\rightarrow\infty}  C_{d,\delta}n^{-2/\delta}\left(\frac{n^{1-2/d}}{s_n^2}\right)^{(2+\delta)/2}=0
\end{align*}
is satisfied. Finally, using the Lyapunov CLT, we obtain \eqref{eq:HilbertCLT}.
\end{proof}

From the proof of Theorem 4 in \cite{he:owen:2016}, we find that 
\begin{equation}\label{eq:upperbd}
\var(\hat{\mu}_{n})\le 4M^2(d+3)n^{-1-2/d}
\end{equation}
for any Lipschitz function $f$ with modulus $M$. Theorem~\ref{thm:HilbertCLTsmooth} gives an asymptotic lower bound of order $n^{-1-2/d}$ for $\var(\hat{\mu}_{n})$. Therefore, the rate $O(n^{-1-2/d})$ is tight for $\var(\hat{\mu}_{n})$ if $f\in C^1([0,1]^d)$ and $\int_{[0,1]^{d}}\Vert\nabla f(X)\Vert^{2}\mrd X>0$. To prove the asymptotic normality, we only require the lower bound as shown in the proof of Theorem~\ref{thm:HilbertCLTsmooth}.

\subsection{Piecewise Smooth Functions}\label{sec:disc}
In this subsection, we focus on piecewise smooth functions of the form $f(X)=g(X)1_{\Omega}(X)$, where $\partial\Omega$ admits a $(d-1)$-dimensional Minkowski content defined below. This kind of functions was also studied in \cite{he:owen:2016}.
\begin{definition}
	For a set $\Omega\subset[0,1]^d$, define
	\begin{equation}\label{eq:fmc}
	\mathcal{M}(\partial \Omega)=\lim_{\epsilon\downarrow 0}\frac{\lambda_d((\partial \Omega)_\epsilon)}{2\epsilon},
	\end{equation}
	where $(A)_\epsilon:=\{x+y|x\in A,\Vert y\Vert\leq \epsilon\}$.
	If $\mathcal{M}(\partial \Omega)$ exists and finite, then $\partial \Omega$ is said to admit a $(d-1)$-dimensional Minkowski content.
\end{definition}

In the terminology of geometry, $\mathcal{M}(\partial \Omega)$  is known as the surface area of the set $\Omega$. The Minkowski content has a clear intuitive basis, compared to the Hausdorff measure that provides an alternative to quantify the surface area. We should note that the Minkowski content coincides with the Hausdorff measure, up to a constant factor, in regular cases.  It is known that the boundary of any convex set in $[0,1]^d$ has a $(d-1)$-dimensional Minkowski content since the surface area of a convex set  in $[0,1]^d$ is bounded by the surface area of  the unit cube $[0,1]^d$, which is $2d$. More generally, \cite{Ambr:2008} found that  $\partial \Omega$ admits a ($d-1$)-dimensional Minkowski content when $\Omega$ has a Lipschitz boundary. 

Let
\begin{align*}
&\mathcal{T}_{\mathrm{int}}=\{1\leq i\leq n|E_{i}\subset\Omega\},
\\
&\mathcal{T}_{\mathrm{bdy}}=\{1\leq i\leq n|E_{i}\cap\Omega\neq\emptyset\}
\backslash\mathcal{T}_{\mathrm{int}},
\end{align*}
be the indices of collections of $E_i$ that are interior to $\Omega$ and at the boundary of $\Omega$, respectively. Denote $|A|$ as the cardinality of the set $A$. 

\begin{lemma}\label{lem:bdy}
If $\partial\Omega$ admits a $(d-1)$-dimensional
Minkowski content, then $|\mathcal{T}_{\mathrm{bdy}}|=O(n^{1-1/d})$.
\end{lemma}
\begin{proof}
The proof is given in the proof of Theorem 4 in \cite{he:owen:2016}. We provide here for completeness.

From Lemma~\ref{lem:ei}, the diameter of $E_i$, denoted by $r_i$, satisfies $r_i\leq 2\sqrt{d+3}n^{-1/d}$. Let $\epsilon=2\sqrt{d+3}n^{-1/d}$. 
From~\eqref{eq:fmc}, for any fixed $\delta>2\mathcal{M}(\partial \Omega)$, there exists $\epsilon_0>0$ such that $\lambda_d((\partial \Omega)_\epsilon)<\delta \epsilon$ whenever $\epsilon<\epsilon_0$. Assume that $n>(2\sqrt{d+3}/\epsilon_0)^d$.  Thus $r_i\le \epsilon<\epsilon_0$. Note that $\cup_{i\in \mathcal{T}_{\mathrm{bdy}}}E_i\subset (\partial \Omega)_\epsilon$. This leads to
$$|\mathcal{T}_{\mathrm{bdy}}|\leq \frac{\lambda_d((\partial \Omega)_\epsilon)}{\lambda_d(E_i)}\leq \frac{\delta\epsilon}{n^{-1}}=2\sqrt{d+3}\delta n^{1-1/d},$$
which completes the proof.
\end{proof}

\begin{theorem}\label{thm:disCLT}
Let $f(X)=g(X)1_{\Omega}(X)$, where
$g(X)\in C^1([0,1]^d)$, $\Omega\subset[0,1]^d$ and $\partial\Omega$ admits a $(d-1)$-dimensional
Minkowski content. Suppose that $\sigma_\Omega^2=\int_{\Omega}\Vert\nabla g(X)\Vert^{2}\mrd X>0$. Then for all sufficiently large $n$,
\begin{equation}\label{eq:lbdis}
\var(\hat{\mu}_{n})\geq \frac{\sigma_\Omega^2}{96}2^{-2/d-d}n^{-1-2/d},
\end{equation}
If $d>2$,
\begin{equation}\label{eq:disHilbertCLT}
\frac{\hat{\mu}_{n}-\mu}{\sqrt{\var(\hat{\mu}_{n})}}\to N(0,1),
\end{equation}
in distribution as $n=b^m\to\infty$.
\end{theorem}
\begin{proof}
Following the notations in the proof of Theorem~\ref{thm:HilbertCLTsmooth},
we have
\begin{equation}\label{eq:decom}
s_{n}^{2}=\sum_{i=1}^{n}\sigma_{i}^{2}
=\sum_{i\in\mathcal{T}_{\mathrm{int}}}\var(g(X^{(i)}))
+\sum_{i\in\mathcal{T}_{\mathrm{bdy}}}\var(g(X^{(i)})1_{\Omega}(X^{(i)})).
\end{equation}
Similar to the proof of Theorem~\ref{thm:HilbertCLTsmooth} for $g\in C^1([0,1]^d)$, we have
\begin{equation*}
\var(g(X^{(i)})) = \frac{1}{12\cdot 2^{2m}}\Vert\nabla g(c_i)\Vert^{2}+o(2^{-2m}),
\end{equation*}
where $c_i\in E_i$ as defined there for $i\in \mathcal{T}_{\mathrm{int}}$.
From \eqref{eq:decom}, we find that 
\begin{equation}
s_{n}^{2}\geq\sum_{i\in\mathcal{T}_{\mathrm{int}}}\var(g(X^{(i)}))=  \frac{1}{12\cdot 2^{2m}}\sum_{i\in\mathcal{T}_{\mathrm{int}}}\Vert\nabla g(c_i)\Vert^{2}+o(2^{-2m}\vert\mathcal{T}_{\mathrm{int}}\vert).
\end{equation}
Note that
\begin{align}
\int_{\Omega}\Vert\nabla g(X)\Vert^{2}\mrd X&=\lim_{n\to\infty}\frac1n\sum_{i=1}^{n}\Vert\nabla g(c_i)\Vert^{2}1_\Omega(c_i)\notag\\&=\lim_{n\to\infty}\frac1n\left(\sum_{i\in\mathcal{T}_{\mathrm{int}}}\Vert\nabla g(c_i)\Vert^{2}+\sum_{i\in\mathcal{T}_{\mathrm{bdy}}}\Vert\nabla g(c_i)\Vert^{2}1_\Omega(c_i)\right)\label{eq:nonbdy}\\
&=\lim_{n\to\infty}\frac1n\sum_{i\in\mathcal{T}_{\mathrm{int}}}\Vert\nabla g(c_i)\Vert^{2},\notag
\end{align}
where we picked $c_i\in E_i\backslash\Omega\neq\emptyset$ for $i\in\mathcal{T}_{\mathrm{bdy}}$ so that the last term of \eqref{eq:nonbdy} is actually zero.
Therefore, similar to \eqref{eq:lowerbound}, we have
\begin{equation}\label{eq:lowerbound2}
\liminf_{n\to\infty} \frac{s_{n}^{2}}{n^{1-\frac{2}{d}}}\ge\frac{1}{96\cdot 2^{2/d+d}}\int_{\Omega}\Vert\nabla g(X)\Vert^{2}\mrd X>0,
\end{equation}
that establishes \eqref{eq:lbdis}.

On the other hand, 
\begin{align*}
&\sum_{i=1}^{n}\mathbb{E}\left[\left|f(X^{(i)})-\mathbb{E}[f(X^{(i)})]\right|^{2+\delta}\right]
\\
&=\sum_{i\in\mathcal{T}_{\mathrm{int}}}\mathbb{E}\left[\left|g(X^{(i)})-\mathbb{E}[g(X^{(i)})]\right|^{2+\delta}\right]
+\sum_{i\in\mathcal{T}_{\mathrm{bdy}}}\mathbb{E}\left[\left|g(X^{(i)})1_{\Omega}(X^{(i)})-\mathbb{E}[g(X^{(i)})1_{\Omega}(X^{(i)})]\right|^{2+\delta}\right].
\end{align*}
Again, for any $\delta>0$,
\begin{equation*}
\mathbb{E}[\vert g(X^{(i)})-\mathbb{E}[g(X^{(i)})]\vert^{2+\delta}]
\leq C_{d,\delta}n^{-\frac{2+\delta}{d}},
\end{equation*}
where $C_{d,\delta}>0$ is some constant that only depends on $d$ and $\delta$. 
This leads to
\begin{equation} 
\sum_{i\in\mathcal{T}_{\mathrm{int}}}\mathbb{E}[|g(X^{(i)})-\mathbb{E}[g(X^{(i)})]|^{2+\delta}]\leq 
C_{d,\delta}n^{1-\frac{2+\delta}{d}}.
\end{equation}
It follows from \eqref{eq:lowerbound2} that
\begin{equation}\label{eq:cond1}
\limsup_{n\rightarrow\infty}\frac{1}{s_{n}^{2+\delta}}
\sum_{i\in\mathcal{T}_{\mathrm{int}}}\mathbb{E}\left[\left|g(U_{i})-\mathbb{E}[g(U_{i})]\right|^{2+\delta}\right]
\le \limsup_{n\rightarrow\infty}\frac{C_{d,\delta}n^{1-\frac{2+\delta}{d}}}{s_{n}^{2+\delta}}=0.
\end{equation}
By the continuity of $g$, there is a constant $D$
with $|g(X)|\leq D$ for all $X\in[0,1]^{d}$. 
Therefore,
\begin{equation}
\sum_{i\in\mathcal{T}_{\mathrm{bdy}}}\mathbb{E}\left[\left|g(X^{(i)})1_{\Omega}(X^{(i)})-\mathbb{E}[g(X^{(i)})1_{\Omega}(X^{(i)})]\right|^{2+\delta}\right]
\leq (2D)^{2+\delta}|\mathcal{T}_{\mathrm{bdy}}|.
\end{equation}
By Lemma~\ref{lem:bdy}, we have $|\mathcal{T}_{\mathrm{bdy}}|=O(n^{1-1/d})$. As a result,
\begin{align}
&\limsup_{n\rightarrow\infty}\frac{1}{s_{n}^{2+\delta}}
\sum_{i\in\mathcal{T}_{\mathrm{bdy}}}\mathbb{E}\left[\left|g(X^{(i)})1_{\Omega}(X^{(i)})-\mathbb{E}[g(X^{(i)})1_{\Omega}(X^{(i)})]\right|^{2+\delta}\right]
\\
&\leq\limsup_{n\rightarrow\infty}
\frac{(2D)^{2+\delta}|\mathcal{T}_{\mathrm{bdy}}|}{s_{n}^{2+\delta}}
\notag\\
&=\limsup_{n\rightarrow\infty}(2D)^{2+\delta}n^{\frac{1+\delta}d-\frac\delta 2}\left(\frac{n^{1-\frac 2d}}{s_{n}^2}\right)^{\frac{2+\delta}2}\frac{|\mathcal{T}_{\mathrm{bdy}}|}{n^{1-1/d}}=0,\label{eq:cond2}
\end{align}
provided that $d>2$ and $\delta>1$. 
Together with \eqref{eq:cond1} and \eqref{eq:cond2}, the Lyapunov condition is thus verified. So the asymptotic normality is satisfied by applying the Lyapunov CLT again.
\end{proof}

\cite{he:owen:2016} gave an upper bound of $O(n^{-1-1/d})$ for $\var(\hat{\mu}_{n})$ if $f(X)=g(X)1_{\Omega}(X)$, where $g$ is Lipschitz continuous. Theorem~\ref{thm:disCLT} provides a lower bound of order $n^{-1-2/d}$. For discontinuous integrands, we cannot get asymptotically
matching lower bound to the upper bound because when we take $\Omega=[0,1]^d$, the lower bound \eqref{eq:lbdis} is in line with the smooth case.  To establish the asymptotic normality, Theorem~\ref{thm:disCLT} requires $d>2$.  It is not clear in general whether the  asymptotic normality holds for $d=1,2$. If the last term of  \eqref{eq:decom} has a lower bound of $O(n^{1-1/d})$, one would have the asymptotic normality  for $d=2$. For $d=1$, let's consider the function $f(X)=g(X)1_{\{X>\theta\}}(X)$ for some $\theta \in [0,1]$. If $\theta$ is a multiple of $b^{-m_0}$ for some $m_0>0$, then the error over the set $\mathcal{T}_{\mathrm{bdy}}$ vanishes whenever $m\ge m_0$.  The Lyapunov condition is thus verified by \eqref{eq:cond1}. As a result, the asymptotic normality holds for this case. If $\theta$ does not have a terminating $b$-adic representation, we may require some additional conditions to ensure the asymptotic normality. Theorem~\ref{thm:disCLT} also requires that $\int_{\Omega}\Vert\nabla g(X)\Vert^{2}\mrd X>0$. That condition actually rules out the case in which $f$ is an indicator function.
The analysis of indicator functions is presented in the next subsection.

\subsection{Indicator Functions}\label{sec:ind}
We now consider indicator functions of the form $f(X)=1_{\Omega}(X)$. Recall that $\mathcal{T}_{\mathrm{bdy}}$ denotes the index of the collections of $E_i$ that touch the boundary of $\Omega$. In this case, the variance of the estimate reduces to 
\begin{equation}\label{eq:indvar}
\var(\hat{\mu}_n) = \frac{1}{n^2}\sum_{i\in \mathcal{T}_{\mathrm{bdy}}}\var(1_{\Omega}(X^{(i)})),
\end{equation}
where $\var(1_{\Omega}(X^{(i)}))=n\lambda_d(E_{i}\cap\Omega)(1-n\lambda_d(E_{i}\cap\Omega))$.
Motivated by the proof of Theorem~\ref{thm:disCLT}, one needs to derive a suitable lower bound for $s_n^2=n^2\var(\hat{\mu}_n)$ to apply the Lyapunov CLT. Note that $s_n^2\leq |\mathcal{T}_{\mathrm{bdy}}|/4$. Assume that $\partial\Omega$ admits a $(d-1)$-dimensional
Minkowski content;  $s_n^2$ then has an upper bound of $O(n^{1-1/d})$ since $|\mathcal{T}_{\mathrm{bdy}}|=O(n^{1-1/d})$ by Lemma~\ref{lem:bdy}. It is easy to see that if $s_n^2\ge c n^{1-1/d}$ for some constant $c>0$, the Lyapunov condition is satisfied for any $d>1$. However,  it is possible that $\mathcal{T}_{\mathrm{bdy}}=\emptyset$ for strictly increasing sample sizes $n_k$, $k=1,\dots,\infty$, if $\Omega$ is a cube. This leads to an identically zero variance and hence $s_{n_k}^2=0$.
Therefore, to study the asymptotic normality for indicator functions, we need the following assumption on $\Omega$, instead of the Minkowski content condition.

\begin{assumption}\label{assump:ind}
For $\Omega\subset [0,1]^d$, there exist a constant $c>0$ and an $N_0\ge 1$ such that for any $n\geq N_0$,
\begin{equation}\label{eq:asump:ind1}
\inf_{i\in\mathcal{T}_{\mathrm{bdy}}}\var(1_{\Omega}(X^{(i)}))\ge c.
\end{equation}
Moreover, 
\begin{equation}\label{eq:asump:ind2}
\lim_{n\rightarrow\infty}|\mathcal{T}_{\mathrm{bdy}}|=\infty.
\end{equation}

\end{assumption}
\begin{theorem}\label{thm:indCLT}
	Let $f(X)=1_{\Omega}(X)$, where $\Omega$ satisfies Assumption~\ref{assump:ind}. Then
	\begin{equation}\label{eq:indHilbertCLT}
	\frac{\hat{\mu}_{n}-\mu}{\sqrt{\var(\hat{\mu}_{n})}}\to N(0,1),
	\end{equation}
	in distribution as $n=b^m\to\infty$.
\end{theorem}
\begin{proof}
By \eqref{eq:indvar} and \eqref{eq:asump:ind1}, we have $s_n^2\ge c |\mathcal{T}_{\mathrm{bdy}}|$. The Lyapunov condition
\begin{align*}
\lim_{n\rightarrow\infty}\frac{1}{s_{n}^{2+\delta}}\sum_{i=1}^{n}\mathbb{E}[\vert f(X^{(i)})-\mathbb{E}[f(X^{(i)})]\vert^{2+\delta}]
&\leq\limsup_{n\rightarrow\infty}\frac{2^{2+\delta}|\mathcal{T}_{\mathrm{bdy}}|}{(c |\mathcal{T}_{\mathrm{bdy}}|)^{(2+\delta)/2}}\\
&=\limsup_{n\rightarrow\infty}  \left(\frac{2}{\sqrt{c}}\right)^{2+\delta}|\mathcal{T}_{\mathrm{bdy}}|^{-\delta/2}=0
\end{align*}
is satisfied for any $\delta>0$, where we used the condition \eqref{eq:asump:ind2} and $c>0$. 
Applying the Lyapunov CLT, we obtain \eqref{eq:indHilbertCLT}.
\end{proof}

Note that for $d=1$, the condition \eqref{eq:asump:ind2} does not hold if $\Omega$ is a  union of $k$ disjoint intervals in [0,1], where $k$ is a given positive integer. This is because $|\mathcal{T}_{\mathrm{bdy}}|\le 2k$ for all possible $n$. Actually, for such cases, the CLT does not hold since the integration error  is distributed over (at most) $k+1$ possible values for any $n=b^m$; see also \cite{lecu:2010} for discussions on randomly-shifted lattice rules.

Define $A_n(c)=\{i\in \mathcal{T}_{\mathrm{bdy}}|\var(1_{\Omega}(X^{(i)}))\ge c\}$.
Assumption~\ref{assump:ind} can be weakened slightly to that
there exist $c>0$ and $\delta>0$ such that
$$\limsup_{n\rightarrow\infty}  \frac{|\mathcal{T}_{\mathrm{bdy}}|}{|A_n(c)|^{1+\delta}}=0.$$ If  $\partial\Omega$ admits a $(d-1)$-dimensional
Minkowski content additionally, it suffices to verify
$$\limsup_{n\rightarrow\infty}  n^{1-1/d}|A_n(c)|^{-1-\delta}=0,$$
or equivalently, $|A_n(c)|^{-1}=o(n^{(1/d-1)/(1+\delta)})$.
This is because $|\mathcal{T}_{\mathrm{bdy}}|=O(n^{1-1/d})$.

However, it may be hard to verify Assumption~\ref{assump:ind} for general $\Omega$. As an illustrative example, we next show that the assumption holds for the case  $\Omega=\{X=(X_1,X_2)\in[0,1]^2|X_1+X_2\ge 1\}$. We restrict our attention to the van der Corput sequence in base $b=2$ so that $n=2^m$. In this case, $E_i$ is a square  with sides of length  $1/\sqrt{n}$ when $m$ is even; when $m$ is odd, $E_i$ is a rectangle with width $\sqrt{2/n}$ and height $1/\sqrt{2n}$; see Figure~\ref{fig:grids} for illustrations. We thus have
\begin{equation*}
|\mathcal{T}_{\mathrm{bdy}}|=
\begin{cases}
\sqrt{n},&m \text{ is even},\\
\sqrt{2n},&m \text{ is odd}.
\end{cases}
\end{equation*}
Moreover, for all $i\in \mathcal{T}_{\mathrm{bdy}}$, we find that
\begin{equation*}
\var(1_{\Omega}(X^{(i)}))=
\begin{cases}
1/4,&m \text{ is even},\\
3/16,&m \text{ is odd}.
\end{cases}
\end{equation*}
Therefore, Assumption~\ref{assump:ind} is satisfied with $c=3/16$ and $N_0=1$ so that the CLT holds for this example.
Similarly, it is easy to see that the CLT still holds for the set $\Omega=\{X=(X_1,\dots,X_d)\in[0,1]^d|\sum_{i=1}^dX_i\ge d/2\}$ in $d$ dimensions.

\begin{figure}[ht]
	\includegraphics[width = \hsize]{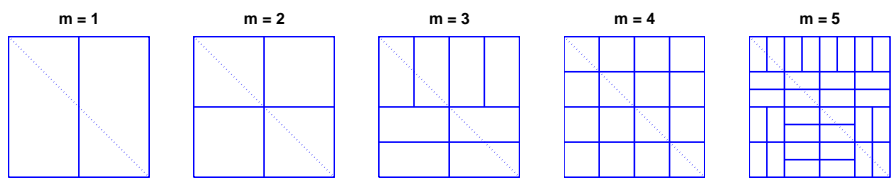}
	\caption{Five splits of $[0,1]^2$ for HSFC stratification, and the dot line is $X_1+X_2=1$.}\label{fig:grids}
\end{figure}

\section{Numerical Results}\label{sec:numer}
In this section, we present some numerical studies to assess the normality of the standardized errors. 
We also examine the lower bound established in Theorem~\ref{thm:HilbertCLTsmooth}  for smooth functions. We consider the integrals of the following functions:
\begin{itemize}
	\item a smooth function, $f_1(X)=12^{d/2}\prod_{i=1}^d(X_i-\frac{1}{2})$,
	\item a piecewise smooth function, $f_2(X)=(X_1-X_2)1_{\{\sum_{i=1}^d X_i\ge d/2\}}(X)$, and
	\item an indicator function, $f_3(X)=1_{\{\sum_{i=1}^d X_i\ge d/2\}}(X)$.
\end{itemize}
Note that for any $d\ge 1$, the exact values of these integrals are $\mu = 0$, $\mu = 0$, and $\mu = 1/2$, respectively. The smooth function was studied in \cite{owen:1997}, which satisfies the smooth condition required in \cite{loh:2003}. The scrambled $(t,m,d)$-net integration of this smooth function has a variance of $O(n^{-3}(\log n)^{d-1})$ \cite{owen:1997}, and it enjoys the asymptotic normality when $t=0$, as confirmed by \cite{loh:2003}.  However, for the last two discontinuous functions, there is no theoretical guarantee in supporting the asymptotic normality for scrambled net quadratures, since the functions do not fit into the class of smooth functions required in \cite{loh:2003}.

We make comparisons with randomized Sobol' points, which use the nested uniform scrambling of \cite{owen:1995} or the linear scrambling of \cite{mato:1998}. We use the C++ library of T. Kollig and A. Keller (\url{http://www.uni-kl.de/AG-Heinrich/Sample
Pack.html}) to generate the nested uniform scrambled Sobol' (NUS--Sobol') points. To generate the linear scrambled Sobol' (LS--Sobol') points, we make use of the generator \verb|scramble| in MATLAB. To calculate Hilbert's mapping function $H(x)$, we use the C++ source code in \cite{lawder:2000}
which is based on the algorithm in \cite{butz:1971}.
To  estimate the variances of these estimators, we use $R$ independent replications $\hat{\mu}_{n}^{(1)},\dots,\hat{\mu}_{n}^{(R)}$ of the sampling schemes. We then estimate the variances by the corresponding empirical variances
\begin{equation*}
\hat{\sigma}^2_n = \frac 1{R-1}\sum_{i=1}^R (\hat{\mu}_{n}^{(i)}-\bar{\mu})^2,
\end{equation*}
where $\bar{\mu} = (1/R) \sum_{i=1}^R \hat{\mu}_n^{(i)}$.
To see asymptotic normality, we plot the kernel smoothed density of the standardized errors
$$Z_i=\frac{\hat{\mu}_n^{(i)}-\mu}{\hat{\sigma}_n},\  i=1,\dots,R,$$
using the function \verb|ksdensity|  in MATLAB. In  our  experiments, we take $R=1000$ and $n=2^{14}=16384$, in order to get good accuracy in the estimation of the target density.

Now consider the smooth function $f_1(X)$. We find that $f_1\in C^1([0,1]^d)$, and
\begin{align*}
\int_{[0,1]^{d}}\Vert\nabla f_1(X)\Vert^{2}\mrd X&=d\int_{[0,1]^{d}}\left(\frac{\partial f(X)}{\partial X_1}\right)^{2}\mrd X\\
&=12^{d}d\int_{[0,1]^{d-1}}\prod_{i=2}^d\left(X_i-\frac{1}{2}\right)^2\mrd X\\
&=12d>0.
\end{align*} 
Therefore, by Theorem~\ref{thm:HilbertCLTsmooth}, the HSFC-based estimate follows the CLT for $f_1$. The lower bound in \eqref{eq:lbsmooth}  becomes $$\var(\hat{\mu}_n)\geq 2^{-3-d-2/d}dn^{-1-2/d}.$$ 
Note that $f_1(X)$ is a Lipschitz function whose modulus $M$ satisfies
$$M\le \sum_{i=1}^d \sup_{X\in[0,1]^d}\left\lvert\frac{\partial f_1(X)}{\partial X_i}\right\rvert=12^{d/2}2^{1-d}d.$$
Together with \eqref{eq:upperbd}, we obtain an upper bound 
$$\var(\hat{\mu}_n)\leq 16(d+3)3^dd^2n^{-1-2/d}.$$ Figure~\ref{fig:bounds} shows the natural logarithm of the empirical variances of the HSFC-based estimator for $n=2^m$, $m=0,\dots,18$. The true variance of  Monte Carlo sampling is $1/n$ for all $d\geq 1$. The lower bound and the upper bound above are also presented. We observe that the  empirical variances decay at the rate $n^{-2}$ for $d=2$, and at the rate $n^{-5/4}$ for $d=8$. This supports that the rate $n^{-1-2/d}$ for the HSFC sampling is tight for smooth functions. Figure~\ref{fig:density1} displays smoothed density estimations of the standardized errors for plain Monte Carlo, LS--Sobol', NUS--Sobol', and HSFC. As expected, a nearly normal distribution appears for both the Monte Carlo and HSFC schemes. For the nested uniform scrambling scheme, a nearly normal distribution is also observed for $d=2$. This is because Sobol' sequence is a  $(t,d)$-sequence in base $b=2$ with $t=0$ for $d=2$ and $t>0$ for $d=8$ \cite{dick:2008}. Therefore, the CLT holds for $d=2$, as confirmed by \cite{loh:2003}. For $d=8$, on the other hand, the density of the standardized errors does not look like a normal distribution. Even worse, for the linear scrambling scheme, the distribution of the standardized errors is very different from the normal distribution. It looks rather spiky for $d=2$.

\begin{figure}[ht]	
	\includegraphics[width=\hsize]{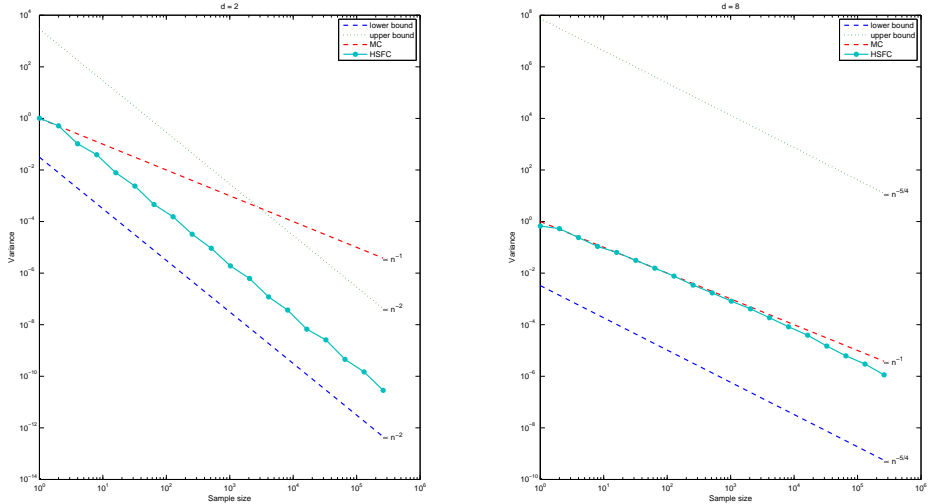}
	\caption{Decay of empirical variance of HSFC sampling as a function of sample size in a log-log scale for $d=2,8$.  The lower bound and the upper bound for the variance are  included. The true variance of Monte Carlo (MC) sampling is also presented for comparison.}\label{fig:bounds}
\end{figure}

\begin{figure}[ht]	
	\includegraphics[width=\hsize]{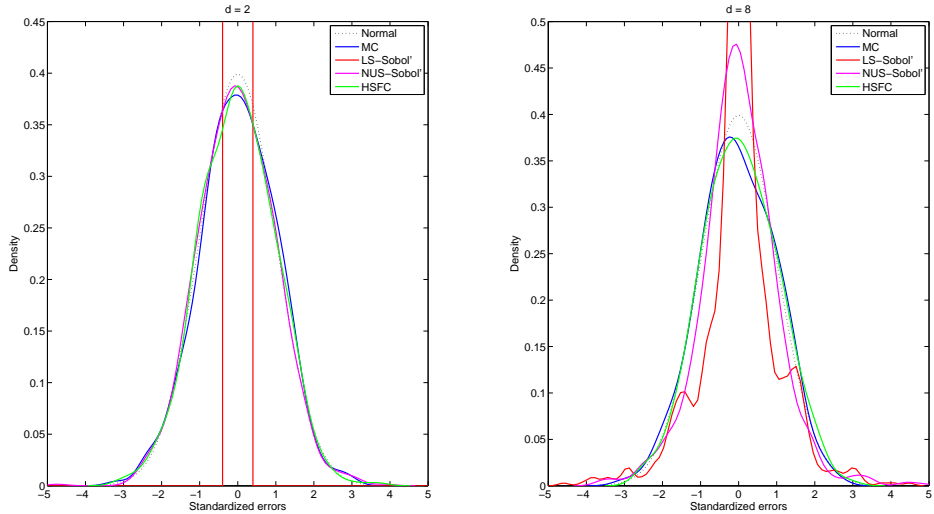}
	\caption{Empirical verification of asymptotic normality for the integrations of the smooth function $f_1(X)$ with plain Monte Carlo (MC), LS--Sobol', NUS--Sobol', and HSFC, the dot curve is the true density of the standard normal $N(0,1)$.}\label{fig:density1}
\end{figure}

For the two discontinuous functions $f_2(X)$ and $f_3(X)$, the CLT holds for the HSFC sampling (see Sections~\ref{sec:disc} and \ref{sec:ind} for details).  Figures~\ref{fig:density2} and \ref{fig:density3} show smoothed density estimations of the standardized errors for the two functions, respectively. As expected, a nearly normal distribution appears for both the Monte Carlo and HSFC schemes with $d=2,8$. More interestingly, a nearly normal distribution is also observed for the nested uniform scrambling scheme in all cases, although it is not clear whether the CLT holds for scrambled net integrations of discontinuous functions. Similar to the case of the smooth function, the integration error distribution for the  linear scrambling scheme is far from the normal distribution, particularly for $d=2$. Comparing to the nested uniform scrambling, the linear scrambling requires less randomness, and therefore its samples may be strongly dependent. That might explain why the CLT does not hold in most cases for randomized QMC with the linear scrambling.

\begin{figure}[ht]	
	\includegraphics[width=\hsize]{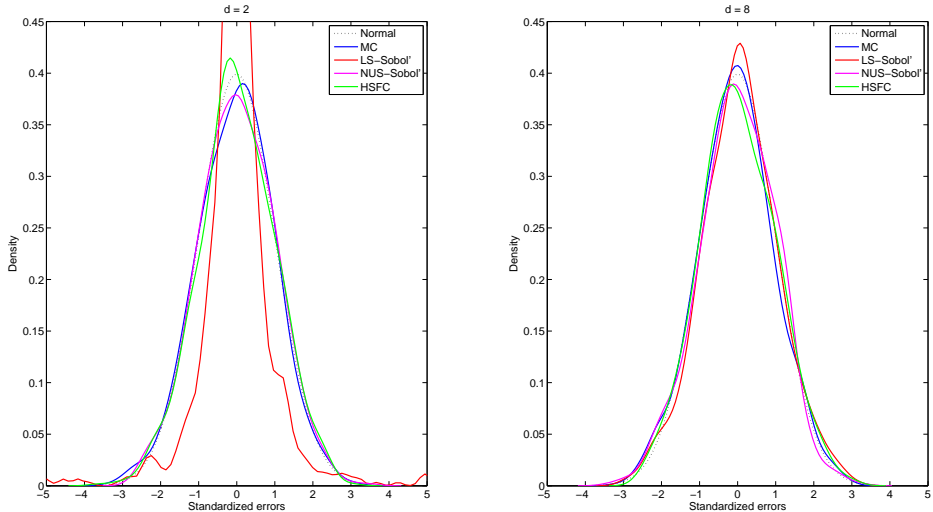}
	\caption{Empirical verification of asymptotic normality for the integrations of the piecewise smooth function $f_2(X)$ with plain Monte Carlo (MC), LS--Sobol', NUS--Sobol', and HSFC, the dot curve is the true density of the standard normal $N(0,1)$.}\label{fig:density2}
\end{figure}

\begin{figure}[ht]	
	\includegraphics[width=\hsize]{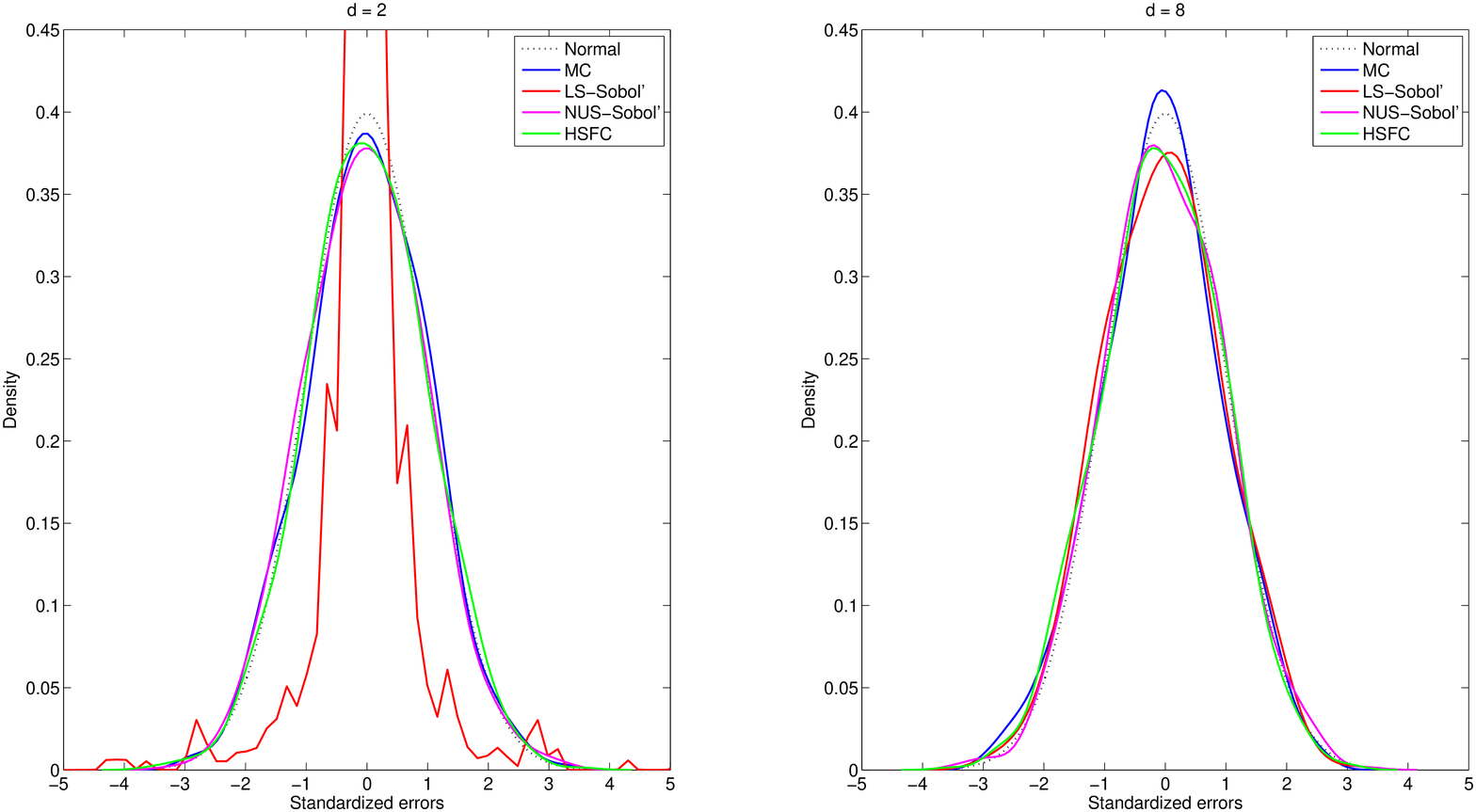}
	\caption{Empirical verification of asymptotic normality for the integrations of the indicator function $f_3(X)$ with plain Monte Carlo (MC), LS--Sobol', NUS--Sobol', and HSFC, the dot curve is the true density of the standard normal $N(0,1)$.}\label{fig:density3}
\end{figure}

\section{Concluding Remarks}\label{eq:final}
\cite{loh:2003} showed that the scrambled net estimate has an
asymptotic normal distribution for certain smooth functions. In a very recent work, \cite{basu:2016} found that the scrambled geometric net estimate has an
asymptotic normal distribution for certain smooth functions defined
on products of suitable subsets of $\Re^d$. The smoothness conditions required in the two papers are more restrictive than the smooth condition required in Section~\ref{sec:smooth}. The proofs in both \cite{loh:2003} and \cite{basu:2016} relied on ensuring a suitable lower bound on the variance of the estimate matching up to constants to the upper bound. The proofs in this paper relies on establishing a suitable lower bound, and then make use of the Lyapunov CLT. 

We also  proved the asymptotic normality of the HSFC-based stratified estimate for certain discontinuous functions. To our best knowledge, it is not clear whether the asymptotic normality of the scrambled net estimate holds for discontinuous functions. \cite{he:wang:2015} provided some upper bounds of scrambled net variances for piecewise smooth functions of the same form  $f(X)=g(X)1_{\Omega}(X)$ studied in Section~\ref{sec:disc}, but $g$ is of bounded variation in the sense of Hardy and Krause instead. For future research, following the procedures in \cite{loh:2003}, it is desirable to establish a matching lower bound for the variance of scrambled net integration of discontinuous functions.

\cite{he:owen:2016} used randomized van der Corput sequence in base $b$ as the input of the HSFC  sampling. This makes the sampling scheme extensible. As in \cite{loh:2003} and  \cite{basu:2016}, the analysis in this paper is based on the sample size with the pattern $n=b^m$, not with arbitrary $n$. This scheme turns out to be a kind of stratified samplings. In contrast to the usual grid sampling, it is extensible and  does not require so highly composite sample sizes, particularly for large $d$. The results on the HSFC sampling can also be applied to the usual grid sampling.

\section*{Acknowledgments}
The authors gratefully thank  Professor Art B. Owen for the helpful comments. 
Zhijian He is supported by the National Science Foundation of
China under Grant 71601189. 
Lingjiong Zhu is supported by the NSF Grant DMS-1613164.

\end{document}